\documentclass[11pt]{article}
\usepackage{amsfonts}
\usepackage{amsmath}
\usepackage{amssymb,amsthm,amscd}
\usepackage{epic, eepic, latexsym}
\usepackage{graphics}
\usepackage{epsfig}
\usepackage{fullpage}
\usepackage{enumerate}
\newcommand{\tr}[0]{\operatorname{tr}}

\newcommand{\Ric}[0]{\operatorname{Ric}}
\newcommand{\Rm}[0]{\operatorname{Rm}}

\newtheorem{theorem}{Theorem}[section]
\newtheorem{lemma}[theorem]{Lemma}

\newtheorem{corollary}[theorem]{Corollary}

\title{Collapsing of Products Along the K\"ahler-Ricci Flow}
\author{Matthew Gill}
\date{}

\begin{document}
\maketitle

\bigskip

\begin{abstract}
Let $X = M \times E$ where $M$ is an $m$-dimensional K\"ahler manifold with negative first Chern class and $E$ is an $n$-dimensional complex torus. We obtain $C^\infty$ convergence of the normalized K\"ahler-Ricci flow on $X$ to a K\"ahler-Einstein metric on $M$. This strengthens a convergence result of Song-Weinkove and confirms their conjecture.
\end{abstract}

\bigskip

\section{Introduction}

Let $M$ be an $m$-dimensional K\"ahler manifold with negative first Chern class and let $E$ be an $n$-dimensional complex torus. Independently from Yau and Aubin, there exists a unique K\"ahler-Einstein metric $g_M$ on $M$ \cite{Yau, Au}. Fix a flat metric $g_E$ on $E$. Recall that we can associate a $(1,1)$-form $\omega$ to a K\"ahler metric $g$ by defining
\begin{equation}
\omega = \frac{\sqrt{-1}}{2\pi} g_{i \bar{j}} dz^i \wedge dz^{\bar{j}}.
\end{equation}
Throughout this paper, we will relate K\"ahler metrics $g$, $g_M$, \ldots with their K\"ahler forms $\omega$, $\omega_M$, \ldots using the obvious notation. We will also refer to $\omega$ as a K\"ahler metric since $\omega$ and $g$ uniquely determine each other. Additionally, a uniform constant $C, C', \ldots$ will be a constant depending only on the initial data whose definition my change from line to line.

Let $X = M \times E$ and define projection maps $\pi_M : X \to M$ and $\pi_E : X \to E$. Let $\omega_0$ be any K\"ahler metric on $X$ and consider the normalized K\"ahler-Ricci flow

\begin{equation} \label{nkrf}
\frac{\partial}{\partial t} \omega = -\Ric\left(\omega\right) - \omega, \ \ \ \ \ \omega_{t=0} = \omega_0.
\end{equation}
Observe that $$\Ric\left(\pi^*_M \omega_M + \pi^*_E \omega_E \right) = -\pi^*_M \omega_M.$$
Hence $c_1\left(X\right) = - [ \pi^*_M \omega_M ] \leq 0$ and the flow \eqref{nkrf} exists for all time by the work of Tsuji \cite{Ts} and Tian-Zhang \cite{TZ}. Notice that in general $\omega_0$ is not a product. In the case when $\omega_0$ is a product, the work of Cao shows that the flow exists for all time and converges smoothly to a K\"ahler-Einstein metric on $M$ \cite{Cao}. We prove the following theorem.

\begin{theorem}
Let $\omega(t)$ be the solution to the normalized K\"ahler-Ricci flow \eqref{nkrf} with initial K\"ahler metric $\omega_0$ on $X = M \times E$. Then 

\begin{enumerate}[(a)]
\item $\omega\left(t\right)$ converges to $\pi^*_M \omega_M$ in $C^\infty \left(X , \omega_0 \right)$ as $t \to \infty$.
\item For any $z \in M$, let $E(z) = \pi^{-1}_M (z)$ denote the fiber above $z$. Then $ e^{t}\omega\left(t\right)|_{E(z)}  \to \omega_{flat}|_{E(z)}$ in $C^\infty \left(E(z), \omega_E\right)$ as $t \to \infty$, where $\omega_{flat}$ is a $(1,1)$-form on $X$ with $[\omega_{flat}] = [\omega_0]$ whose restriction to each fiber is a flat K\"ahler metric.
\end{enumerate}

\end{theorem}

We remark that this theorem holds for any compact K\"ahler manifold that admits a flat metric, which includes certain quotients of complex tori. This theorem strengthens a convergence result of Song and Weinkove and confirms their conjecture \cite{SW3}. They prove that when $m = n = 1$, the convergence in $(a)$ takes place in $C^\beta(X, \omega_0)$ for any $\beta$ between $0$ and $1$, and that the convergence in $(b)$ takes place in $C^0\left(E(z), \omega_E \right)$. They conjecture that the convergence in this case is in fact $C^\infty$. This problem originates from the work of Song and Tian \cite{ST1}. They considered the normalized K\"ahler-Ricci flow on an elliptic surface $f : X \to \Sigma$ where some of the fibers may be singular. It was shown that the solution of the flow converges to a generalized K\"ahler-Einstein metric on the base $\Sigma$ in $C^{1,1}$. This result was generalized to the fibration $f : X \to X_{can}$ where $X$ is a nonsingular algebraic variety with semi-ample canonical bundle and $ X_{can}$ is its canonical model \cite{ST2}. Theorem 1.1 is a step towards strengthening this convergence result to $C^\infty$. We remark that Gross, Tosatti and Zhang have studied a similar manifold as in Theorem 1.1, but considered the case where the K\"ahler class of the metric tends to the boundary of the K\"ahler cone instead of evolving by the K\"ahler-Ricci flow \cite{GTZ}. Fong and Zhang have examined the rate of collapse of the fibers of a similar manifold along the K\"ahler-Ricci flow in a recent preprint \cite{FZ}.

Theorem 1.1 is related to viewing the K\"ahler-Ricci flow with surgery as an analytic Minimal Model Program (MMP) as conjectured by Song and Tian and proved in the weak sense \cite{ST3}. The idea of the MMP is that after several blow-downs and flips, a projective algebraic variety becomes either a minimal model or a Mori fiber space (an algebraic fibration $f: X \to B$ where the generic fibers are Fano). Recent results due to Song and Weinkove show that the K\"ahler-Ricci flow performs blow-downs as canonical surgical contractions in complex dimension 2 \cite{SW1} and in the case of the blow-up of orbifold points \cite{SW2}. Song and Yuan have given an example of the flow performing a flip \cite{SY}. Specific examples of collapsing along the flow have been investigated by Song and Weinkove in the case of a Hirzebruch surface \cite{SW0} and by Fong in the case of a projective bundle over a K\"ahler-Einstein manifold \cite{F1}.

After performing blow-downs and flips, the K\"ahler-Ricci flow is conjectured to produce either a minimal model or a Mori fiber space. If we continue the flow on a Mori fiber space, the flow is expected to collapse the fibers in finite time. An example of this was examined by Song, Sz\'{e}kelyhidi and Weinkove \cite{SSW}. The rate of collapse of the diameter was improved by Fong under an assumption on the Ricci curvature \cite{F2}. If we continue the flow on a minimal model, the flow exists for all time because the canonical class is nef. In this case, the rescaled flow may collapse in infinite time. This is the case considered in \cite{ST1, ST2, SW3, FZ} and in this paper. 

In section 2, we derive several estimates following \cite{SW3}. Section 3 contains new higher order estimates for the case of a degenerating metric using only the maximum principle. If the metric is not degenerating, then the work in section 3 most likely gives an alternate proof of the results in \cite{ShW}. For other examples of where higher order estimates were obtained using only the maximum principle, see \cite{Ch, DH, LSY}. In section 4, we obtain the convergence of $\omega$, completing the proof of the main theorem. 

\section{Estimates}

 First we establish reference metrics and reduce the flow to a parabolic complex Monge-Amp\`ere equation. The K\"ahler class of $\omega$ evolves as $$[\omega(t)] = e^{-t} [\omega_0] + \left( 1 - e^{-t} \right) [\omega_M].$$ This can be verified by substituting in to the normalized K\"ahler-Ricci flow. Note that we have written $\omega_M$ in place of $\pi^*_M \omega_M$ to simplify notation and we will continue to do so for the remainder of this paper. 

We define a family of reference metrics $\hat{\omega}_t$ in the class of $\omega(t)$ by $$\hat{\omega}_t = e^{-t} \omega_0 + \left( 1 - e^{-t} \right)  \omega_M.$$ Pick a smooth volume form $\Omega$ on $X$ such that 
\begin{equation}\label{eqnOmega}
\frac{\sqrt{-1}}{2\pi} \partial \bar{\partial} \log \Omega =  \omega_M, \ \ \ \int_X \Omega = \tbinom{m+n}{m} \int_X  \omega_M^m \wedge \omega_0^n.
\end{equation} 
This is possible since $\omega_M$ represents the negative of the first Chern class of $X$. Consider the parabolic complex Monge-Amp\`ere equation

\begin{equation}\label{pcma}
\frac{\partial}{\partial t} \varphi = \log \frac{e^{nt} \left(\hat{\omega}_t + \frac{\sqrt{-1}}{2\pi} \partial \bar{\partial} \varphi \right)^{m+n}}{\Omega} - \varphi, \ \ \ \ \hat{\omega}_t + \frac{\sqrt{-1}}{2\pi} \partial \bar{\partial} \varphi > 0, \ \ \ \ \varphi_{t=0} = 0.
\end{equation}
Then the solution $\varphi$ to \eqref{pcma} exists for all time and $\omega(t) = \hat{\omega}_t + \frac{\sqrt{-1}}{2\pi} \partial \bar{\partial} \varphi$ solves the normalized K\"ahler-Ricci flow \eqref{nkrf}.

We derive uniform estimates for the K\"ahler potential $\varphi$. The result of Lemma \ref{lemmaPhiBounds} and Lemma \ref{lemmaUnifEquiv} were proved in more general settings in the work of Song and Tian \cite{ST1}. See also \cite{FZ} in the case of a holomorphic submersion $X \to \Sigma$. Following the notation in \cite{SW3}, we provide a proof for the reader's convenience. 

\begin{lemma}\label{lemmaPhiBounds}
There exists $C > 0$ such that $X \times [0,\infty)$,
\begin{enumerate}[(a)]
\item $\left| \varphi \right| \leq C.$
\item $\left| \dot{\varphi} \right| \leq C.$
\item $\frac{1}{C} \hat{\omega}_t^{m+n} \leq \omega^{m+n} \leq C \hat{\omega}^{m+n}_t.$
\end{enumerate}
\end{lemma}

\begin{proof}

We begin by calculating 
\begin{equation}
e^{nt}\hat{\omega}^{m+n}_t   =  e^{-mt} \omega^{m+n}_0 + \tbinom{m+n}{1}e^{-(m-1)}\left(1 - e^{-t}\right) \omega^{m+n-1}_0 \wedge \omega_M + \ldots + \tbinom{m+n}{m}\left(1 - e^{-t} \right)^m \omega^n_0 \wedge \omega^m_M.
\end{equation}
This equation implies that
\begin{equation}\label{volBound}
\frac{1}{C} \Omega\leq e^{nt} \hat{\omega}^{m+n}_t \leq C \Omega.
\end{equation}

To obtain the upper bound for $\varphi$, assume that $\varphi$ attains a maximum at a point $(z_0, t_0)$ with $t_0 > 0$. At that point, the maximum principle implies 
\begin{equation}
0 \leq \frac{\partial}{\partial t} \varphi \leq \log \frac{e^{nt}\hat{\omega}^{m+n}_t}{\Omega} - \varphi \leq \log C - \varphi.
\end{equation}
Thus we find $\varphi \leq \log C$, giving the upper bound. Similarly, we obtain a lower bound giving (a).

To prove (b), we calculate the evolution equation of $\dot{\varphi}$ to be 
\begin{equation}\label{evoPhiDot}
\left( \frac{\partial}{\partial t} - \Delta \right) \dot{\varphi} = \tr_{\omega} \left(\omega_M - \hat{\omega}_t \right) + n - \dot{\varphi}.
\end{equation}
Note that by the definition of $\hat{\omega}_t$ there exists a constant $C_0 > 1$ such that $\omega_M \leq C_0 \hat{\omega}_t$ (however it is not true that there exists $C_0 > 0$ such that $\frac{1}{C_0} \hat{\omega}_t \leq \omega_M$ since $\omega_M$ is degenerate). Then at the maximum of the quantity $Q_1 = \dot{\varphi} - (C_0 - 1)\varphi$,

\begin{eqnarray}
0 & \leq & \left( \frac{\partial}{\partial t} - \Delta \right) Q_1  =  \tr_{\omega}\left(\omega_M - \hat{\omega}_t \right) + n - \dot{\varphi} - \left(C_0 - 1\right) \dot{\varphi} + \left(C_0 - 1 \right) \Delta \varphi \nonumber \\
& \leq & \left(C_0 - 1\right) \tr_{\omega} \hat{\omega}_t + n - C_0 \dot{\varphi} + \left( C_0 - 1 \right) \tr_{\omega} \left(\omega - \hat{\omega}_t\right) \nonumber \\
& \leq & n + \left(C_0 - 1\right)\left(m + n\right) - C_0 \dot{\varphi}. 
\end{eqnarray}
Hence $Q_1$ is bounded above, and so is $\dot{\varphi}$ by (a).

To obtain the lower bound for $\dot{\varphi}$, we define the quantity $Q_2 = \dot{\varphi} + \left(m+1\right) \varphi$. Working at a point where $Q_2$ achieves a minimum,
\begin{eqnarray}
0 &\geq& \left(\frac{\partial}{\partial t} - \Delta \right) Q_2  =  \tr_{\omega} \left( \omega_M - \hat{\omega}_t \right) + n - \dot{\varphi} + \left(1 + m\right) \dot{\varphi} - \left( m + 1 \right) \tr_{\omega} \left( \omega - \hat{\omega}_t \right) \nonumber \\
& \geq & m \left( \tr_{\omega} \hat{\omega}_t + \dot{\varphi} - \left( m + n + 1\right) \right).
\end{eqnarray}
Using the arithmetic-geometric mean inequality and \eqref{volBound},
\begin{equation}
e^{-\frac{\left(\dot{\varphi} + \varphi \right)}{m+n}} = \left( \frac{\Omega}{e^{nt} \omega^{m+n}} \right)^{\frac{1}{m+n}} \leq C \left( \frac{\hat{\omega}_t^{m+n}}{\omega^{m+n}} \right)^{\frac{1}{m+n}} \leq C \tr_\omega \hat{\omega}_t \leq C - \dot{\varphi}.
\end{equation}
This gives a uniform lower bound for $\dot{\varphi}$ at $(z_0, t_0)$, and hence a uniform lower bound for $\dot{\varphi}$.

Finally, for (c), using (a), (b) and \eqref{pcma} we have
\begin{equation}
\frac{1}{C} \leq \frac{e^{nt}\omega^{m+n}}{\Omega} \leq C,
\end{equation}
completing the proof of the lemma.
\end{proof}

Recall that we say two metrics $\omega_1$ and $\omega_2$ are uniformly equivalent if there exists a constant $C > 0$ such that $\frac{1}{C}\omega_2 \leq \omega_1 \leq C \omega_2$.  We now show that $\omega$ is uniformly equivalent to $\hat{\omega}_t$. Although the following lemma is known in more generality (see \cite{ST1}, \cite{FZ}), we provide a proof for the reader's convenience. We introduce another family of reference metrics
\begin{equation}
\tilde{\omega}_t = \omega_M + e^{-t} \omega_E.
\end{equation}
By writing $\tilde{\omega}_0 = \omega_M + \omega_E$ and $\tilde{\omega_t} = e^{-t} \tilde{\omega_0} + \left ( 1 - e^{-t} \right) \omega_M$, it is easy to see that $\hat{\omega}_t$ and $\tilde{\omega}_t$ are uniformly equivalent. We choose $\tilde{\omega}_t$ so that its curvature tensor vanishes on $E$ which will be useful for the remainder of this paper.

\begin{lemma}\label{lemmaUnifEquiv}
The metrics $\omega$ and $\tilde{\omega}_t$ are uniformly equivalent, i.e. there exists $C > 0$ such that on $X \times [0, \infty)$,
\begin{equation}
\frac{1}{C} \tilde{\omega}_t \leq \omega \leq C \tilde{\omega}_t.
\end{equation}
\end{lemma}
We remark that since $\hat{\omega}_t$ is uniformly equivalent to $\tilde{\omega}_t$, we also have the following corollary.

\begin{corollary}\label{lemmaUnifEquiv2}
The metrics $\omega$ and $\hat{\omega}_t$ are uniformly equivalent.
\end{corollary}

Now we will prove the above lemma using a method similar to Song and Weinkove. The main difference in the proof is that we need to be careful with the curvature tensor of $\tilde{\omega}_t$ due to the increase in dimension.

\begin{proof}

By Lemma \ref{lemmaPhiBounds} part (c), the lemma will follow by bounding $\tr_{\tilde{\omega}_t} \omega$ from above. We begin with the evolution equation for the quantity $\log \tr_{\tilde{\omega}_t} \omega$ from \cite{SW3}. This is analogous to Cao's \cite{Cao} second order estimate, which is the parabolic version of an elliptic estimate from Yau and Aubin \cite{Yau, Au}:
\begin{equation}\label{eqn4}
\left( \frac{\partial}{\partial t} - \Delta \right) \log \tr_{\tilde{\omega}_t} \omega \leq - \frac{1}{ \tr_{\tilde{\omega}_t} \omega}  g^{ \bar{l}k}{ R(\tilde{g}_t)_{k\bar{l}}}^{\bar{j}i}g_{i\bar{j}}.
\end{equation}

To control the Riemann curvature tensor of $\tilde{g}$, we choose product normal coordinates for $g_M$ and $g_E$. In these coordinates,

\begin{equation}\label{eqn5}
   R(\tilde{g}_t)_{k\bar{l}i\bar{j}} = \left\{
     \begin{array}{cl}
       R(g_M)_{k\bar{l}i\bar{j}} & : 1 \leq i, j, k, l \leq m \\
       0  & : \operatorname{else}
     \end{array}
   \right .
\end{equation}

We recall that an inequality of tensors $T_{k\bar{l}i\bar{j}} \leq S_{k\bar{l}i\bar{j}}$ in the Griffiths sense is defined as follows. For any vectors $X$ and $Y$ of type $T^{1,0}$, we have $T_{k\bar{l}i\bar{j}} X^k \overline{X^l} Y^i \overline{Y^j}\leq S_{k\bar{l}i\bar{j}}  X^k \overline{X^l} Y^i \overline{Y^j}$. Since $\Rm(g_M)$ (the Riemann curvature tensor of $g_M$, $R_{k\bar{l}i\bar{j}}$) is a fixed tensor on $M$, for every $X$ and $Y$ on $M$,
\begin{equation}
\left\vert R(g_M)_{k\bar{l}i\bar{j}}  X^k \overline{X^l} Y^i \overline{Y^j} \right\vert^2_{g_M} \leq \left\vert \Rm(g_M) \right\vert^2_{g_M} \left\vert X \right\vert^2_{g_M} \left\vert Y \right\vert^2_{g_M}.
\end{equation}
This gives the following inequality in the Griffiths sense
\begin{equation}\label{eqn6}
-R(g_M)_{k\bar{l}i\bar{j}} \leq C_1 (g_M)_{k\bar{l}} (g_M)_{i\bar{j}}.
\end{equation}
Applying \eqref{eqn5} and \eqref{eqn6} to \eqref{eqn4} gives 
\begin{eqnarray}\label{evoLogTrace}
\left( \frac{\partial}{\partial t} - \Delta \right) \log \tr_{\tilde{\omega}_t} \omega & \leq & \frac{1}{ \tr_{\tilde{\omega}_t} \omega}  \sum_{i,j,l,k,p,q = 1}^m C_1 g^{ \bar{l}k}g_{i\bar{j}}\tilde{g}_t^{\bar{q}i} \tilde{g}_t^{\bar{j}p} (g_M)_{k\bar{l}} (g_M)_{p\bar{q}} \nonumber \\
& = & C_1 \frac{1}{ \tr_{\tilde{\omega}_t} \omega} \left( \tr_{\omega} \omega_M\right) \sum_{i = 1}^m g_{i \bar{i}} \nonumber \\
& \leq & C_1  \frac{1}{ \tr_{\tilde{\omega}_t} \omega} \left( \tr_{\omega} \omega_M\right) \left(\tr_{\tilde{\omega}_t} \omega \right) \nonumber \\
& = & C _1 \tr_{\omega} \omega_M.
\end{eqnarray}

Recall that there exists $C_0 > 1$ such that $\omega_M \leq C_0 \hat{\omega}_t$. Now we define the quantity $Q_3 = \log \tr_{\tilde{\omega}_t} \omega - (C_0 C_1 + 1)\varphi$. Then at the maximum of $Q_3$,
\begin{eqnarray}
\left( \frac{\partial}{\partial t} - \Delta \right) Q_3 & \leq &  C _1 \tr_{\omega} \omega_M - (C_0 C_1 +1) \dot{\varphi} + (C_0 C_1 + 1) \tr_{\omega}\left(\omega - \hat{\omega}_t\right) \nonumber \\
& \leq & (C_0 C_1 + 1)(m+n) - (C_0 C_1 + 1) \dot{\varphi} - \tr_{\omega} \hat{\omega}_t \nonumber \\
& \leq & C - \frac{1}{C} \tr_{\tilde{\omega}_t} \omega.
\end{eqnarray}
To get the last line we use the fact that $\dot{\varphi}$ is bounded from Lemma \ref{lemmaPhiBounds} part (b), that $\tilde{\omega}_t$ and $\hat{\omega}_t$ are uniformly equivalent, and Lemma \ref{lemmaPhiBounds} part (c). Using Lemma \ref{lemmaPhiBounds} part (a) and the maximum principle shows that $Q_3$ is bounded, hence so is $\tr_{\tilde{\omega}_t} \omega$.
\end{proof}

By choosing product normal coordinates for $g_M$ and $g_E$, $\partial_k (\tilde{g}_t)_{i\bar{j}} = 0$ for all $i$, $j$ and $k$ and for all $t \geq 0$. This implies that the Christoffel symbols for $\tilde{\omega}_t$ do not depend on $t$, hence we may write $\tilde{\nabla}$ for both $\nabla_{\tilde{g}_t}$ and $\nabla_{\tilde{g}_0}$ without ambiguity. This also implies that the curvature tensor ${R(\tilde{g}_t)_{i\bar{j}k}}^l$ does not depend on time. Using these facts, we prove the following lemma which we will make heavy use of for the remainder of the paper. We remark that the proof of the following lemma uses the product structure of the manifold in a very strong way.

\begin{lemma}\label{lemmaTildeCurv}
Let $\Rm(\tilde{g}_0)$ denote the Riemann curvature tensor of $\tilde{g}_0$, ${R(\tilde{g}_0)_{i\bar{j}k}}^l$. Then there exists a uniform $C(k) > 0$ for $k = 0, 1, 2, \ldots$ such that on $X \times [0,\infty)$,
\begin{equation}
\vert \tilde{\nabla}^k_{\mathbb{R}} \Rm(\tilde{g}_0) \vert^2 \leq C(k),
\end{equation}
where $\vert \cdot \vert$ denotes the norm with respect to $g(t)$ and where $\tilde{\nabla}_{\mathbb{R}}$ is the covariant derivative with respect to $\tilde{g}_0$ as a Riemannian metric. 
\end{lemma}

\begin{proof}
Recall that $\tilde{g}_t$ is a product metric on $X = M \times E$. Using the fact that $\Rm(\tilde{g}_t)$ does not depend on time and Lemma \ref{lemmaUnifEquiv},
\begin{equation}
\vert \tilde{\nabla}^k_{\mathbb{R}} \Rm(\tilde{g}_0) \vert^2 =  \vert \nabla^k_{\tilde{g}_t, \mathbb{R}} \Rm(\tilde{g}_t) \vert^2_g  \leq  C  \vert \nabla^k_{\tilde{g}_t, \mathbb{R}} \Rm(\tilde{g}_t) \vert^2_{\tilde{g}_t}.
\end{equation}
Then because $g_E$ is a flat metric on $E$,
\begin{equation}
\vert \tilde{\nabla}^k_{\mathbb{R}} \Rm(\tilde{g}_0) \vert^2  \leq C \vert \nabla^k_{\tilde{g}_t, \mathbb{R}} \Rm(\tilde{g}_t) \vert^2_{\tilde{g}_t} = C \vert  \nabla^k_{g_M, \mathbb{R}} \Rm(g_M) \vert^2_{g_M} \leq C(k).
\end{equation}

\end{proof}

We will now bound the first derivative of the metric $\omega$ following the method of \cite{SW3}.

\begin{lemma}\label{lemmaGradEst}
There exists a uniform $C > 0$ such that on $X \times [0,\infty)$,
\begin{equation}
S := \vert \tilde{\nabla} g \vert^2 \leq C \ \ \ and \ \ \  \vert \tilde{\nabla} g \vert_{\tilde{g}_0}^2 \leq C
\end{equation}
where $\vert \cdot \vert$ and $\vert \cdot \vert_{\tilde{g}_0}$ denote the norms with respect to $g(t)$ and $\tilde{g}_0$ respectively. Moreover, 
\begin{equation}\label{ineqS}
\left(\frac{\partial}{\partial t} - \Delta \right) S \leq - \frac{1}{2}|\Rm (g)|^2 + C'
\end{equation}
for some uniform $C' > 0$ and where $\Rm(g)$ denotes the Riemann curvature tensor of $g$, ${R_{i\bar{j}k}}^l$.
\end{lemma}

\begin{proof}
We will derive the evolution equation of $S$ using a formula of Phong-Sesum-Sturm \cite{PSS}. We follow the notation of \cite{PSS, SW3}. Let $\Psi^k_{ij} = \Gamma^k_{ij} - \tilde{\Gamma}^k_{ij} = g^{\bar{l}k} \tilde{\nabla}_i g_{j\bar{l}}$, where $\Gamma$ and $\tilde{\Gamma}$ are the Christoffel symbols for $g(t)$ and $\tilde{g}_0$ respectively. Then we have 
\begin{equation}
S = \vert \Psi \vert^2 =  g^{\bar{j}i} g^{\bar{l}k} g_{p\bar{q}} \Psi^p_{ik} \overline{\Psi^q_{jl}}.
\end{equation}
Before computing the evolution equation of $S$, we need the evolution equation of $\Psi^k_{ij}$.
\begin{equation}
\frac{\partial}{\partial t} \Psi^k_{ij}  =  \frac{\partial}{\partial t} \left( g^{\bar{l} k} \partial_i g_{jl} - \tilde{g}^{\bar{l} k} \partial_i \tilde{g}_{j \bar{l}}\right) = g^{\bar{l} k} \partial_i \left( -R_{j\bar{l}} - g_{j\bar{l}} \right) = -\nabla_i {R_j}^k.
\end{equation}
We also compute the rough Laplacian of $\Psi_{ij}^k$:
\begin{equation}
\Delta \Psi^k_{ij} = g^{\bar{q}p} \nabla_p \nabla_{\bar{q}} \Psi^k_{ij}= \nabla^{\bar{q}} \left( {R(\tilde{g}_0)_{i\bar{q}j}}^k - {R_{i\bar{q}j}}^k \right) = \nabla^{\bar{q}} {R(\tilde{g}_0)_{i\bar{q}j}}^k - \nabla_i {R_j}^k.
\end{equation}
Hence we have
\begin{equation}\label{evoPsi}
\left( \frac{\partial}{\partial t} - \Delta \right) \Psi^k_{ij} = -\nabla^{\bar{q}} {R(\tilde{g}_0)_{i\bar{q}j}}^k.
\end{equation}
Now we calculate the evolution of $S$.
\begin{eqnarray}\label{timeS}
\frac{\partial}{\partial t} S & = & \frac{\partial}{\partial t} \left( g^{\bar{j}i} g^{\bar{l}k} g_{p\bar{q}} \Psi^p_{ik} \overline{\Psi^q_{jl}}\right) \nonumber \\
& = & -\left(- R^{\bar{j}i} -g^{\bar{j}i}\right)  g^{\bar{l}k} g_{p\bar{q}} \Psi^p_{ik} \overline{\Psi^q_{jl}} - g^{\bar{j}i} \left( -R^{\bar{l}k} - g^{\bar{l}k}\right)  g_{p\bar{q}} \Psi^p_{ik} \overline{\Psi^q_{jl}} \nonumber \\
& & \ \ + g^{\bar{j}i} g^{\bar{l}k} \left( -R_{p\bar{q}} - g_{p\bar{q}} \right) \Psi^p_{ik} \overline{\Psi^q_{jl}} + 2 \operatorname{Re} \left(g^{\bar{j}i} g^{\bar{l}k} g_{p\bar{q}} \left( \Delta \Psi^p_{ik} - \nabla^{\bar{s}} {R(\tilde{g}_0)_{i\bar{s}k}}^p\right) \overline{\Psi^q_{jl}}\right)
\end{eqnarray}
Taking the Laplacian of $S$,
\begin{equation}\label{LaplaceS}
\Delta S  = \vert \nabla \Psi \vert^2 + \vert \bar{\nabla} \Psi \vert^2 + g^{\bar{j}i} g^{\bar{l}k} g_{p\bar{q}} \left( \left(\Delta \Psi^p_{ik} \right) \overline{\Psi^q_{jl}} + \Psi^p_{ik} \overline{\left(\bar{\Delta} \Psi^q_{jl}\right)}\right).
\end{equation}
We have the following commutation formula:
\begin{equation}\label{commuteLaplace1}
 \overline{\left(\bar{\Delta} \Psi^q_{jl}\right)} = \Delta \Psi^q_{jl}+ {R_j}^r \Psi^q_{rl} + {R_l}^r \Psi^q_{jr} - {R_r}^q \Psi^r_{jl}.
\end{equation}
Substituting \eqref{commuteLaplace1} into \eqref{LaplaceS} and combining with \eqref{timeS}, we obtain
\begin{equation}\label{evoS}
\left(\frac{\partial}{\partial t} - \Delta \right) S = S - \vert \nabla \Psi \vert^2 - \vert \bar{\nabla} \Psi \vert^2 - 2 \operatorname{Re} \left( g^{\bar{j}i} g^{\bar{l}k} g_{p\bar{q}}  \nabla^{\bar{s}} {R(\tilde{g}_0)_{i\bar{s}k}}^p \overline{\Psi^q_{jl}}\right)
\end{equation}
Now we need to control the final term in \eqref{evoS} to complete the proof. By choosing normal coordinates for $\tilde{g}_0$,
\begin{eqnarray}\label{eqn90}
 2 \operatorname{Re} \left( g^{\bar{j}i} g^{\bar{l}k} g_{p\bar{q}}  \nabla^{\bar{s}} {R(\tilde{g}_0)_{i\bar{s}k}}^p \overline{\Psi^q_{jl}}\right) & = & 2 \operatorname{Re} \Big( g^{\bar{j}i} g^{\bar{l}k} g_{p \bar{q}} g^{\bar{s}r} \Big(\tilde{\nabla}_r {R(\tilde{g}_0)_{i\bar{s}k}}^p - \Psi^a_{ir}  {R(\tilde{g}_0)_{a\bar{s}k}}^p  \nonumber \\
& & \ \ -\Psi^a_{kr}  {R(\tilde{g}_0)_{i\bar{s}a}}^p + \Psi^p_{ar}  {R(\tilde{g}_0)_{i\bar{s}k}}^a \Big) \overline{\Psi^q_{jl}} \Big).
\end{eqnarray}
We bound the first term in \eqref{eqn90} using Lemma \ref{lemmaTildeCurv}:
\begin{equation}\label{eqn91}
 \left\vert 2 \operatorname{Re} \left(g^{\bar{j}i} g^{\bar{l}k} g_{p \bar{q}} g^{\bar{s}r} \tilde{\nabla}_r {R(\tilde{g}_0)_{i\bar{s}k}}^p \overline{\Psi_{jl}^q} \right) \right\vert\leq C \vert \tilde{\nabla} \Rm(\tilde{g}_0) \vert^2 + CS \leq C + CS.
\end{equation}
Similarly for the remaining terms in \eqref{eqn90},
\begin{equation}\label{eqn92}
 \left\vert 2 \operatorname{Re} \Big( g^{\bar{j}i} g^{\bar{l}k} g_{p \bar{q}} g^{\bar{s}r}  {R(\tilde{g}_0)_{a\bar{s}k}}^p \Psi^a_{ir}\overline{\Psi^q_{jl}} \Big) \right\vert \leq C \vert \Rm(\tilde{g}_0) \vert^2 S \leq CS.
\end{equation} 
Using \eqref{eqn90}, \eqref{eqn91} and \eqref{eqn92}, we obtain the estimate
\begin{equation}\label{eqn93}
\left\vert 2 \operatorname{Re} \left( g^{\bar{j}i} g^{\bar{l}k} g_{p\bar{q}}  \nabla^{\bar{s}} {R(\tilde{g}_0)_{i\bar{s}k}}^p \overline{\Psi^q_{jl}}\right) \right\vert \leq C'+ CS.
\end{equation}
We combine \eqref{eqn93} with \eqref{evoS} to obtain
\begin{equation}\label{evoS2}
\left(\frac{\partial}{\partial t} - \Delta \right) S \leq  C'+ CS - \vert \nabla \Psi \vert^2 - \vert \bar{\nabla} \Psi \vert^2.
\end{equation}

Define the quantity $Q_4 = S + A \tr_{\tilde{\omega}_t} \omega$ where $A$ is a large constant to be determined later. The evolution equation of $\tr_{\tilde{\omega}_t} \omega$ is (see \cite{SW3}),
\begin{eqnarray}\label{eqn7}
\left( \frac{\partial}{\partial t} - \Delta \right)  \tr_{\tilde{\omega}_t} \omega & = & -\tr_{\tilde{\omega}_t} \omega -  g^{ \bar{l}k}{ R(\tilde{g}_t)_{k\bar{l}}}^{\bar{j}i}g_{i\bar{j}} -  g^{\bar{l}k} \tilde{g}_t^{\bar{j}i} g^{\bar{q}p} \tilde{\nabla}_i g_{k\bar{q}} \tilde{\nabla}_{\bar{j}} g_{p \bar{l}} \nonumber \\
& \leq& - g^{ \bar{l}k}{ R(\tilde{g}_t)_{k\bar{l}}}^{\bar{j}i}g_{i\bar{j}} -  g^{\bar{l}k} \tilde{g}_t^{\bar{j}i} g^{\bar{q}p} \tilde{\nabla}_i g_{k\bar{q}} \tilde{\nabla}_{\bar{j}} g_{p \bar{l}}.
\end{eqnarray}

Using \eqref{eqn7} and \eqref{evoS2} we have
\begin{equation}\label{evoQ4}
\left(\frac{\partial}{\partial t} - \Delta \right) Q_4 \leq C' + CS - \vert \nabla \Psi \vert^2 - \vert \bar{\nabla} \Psi \vert^2 - A g^{ \bar{l}k}{ R(\tilde{g}_t)_{k\bar{l}}}^{\bar{j}i}g_{i\bar{j}} - A g^{\bar{l}k} \tilde{g}_t^{\bar{j}i} g^{\bar{q}p} \tilde{\nabla}_i g_{k\bar{q}} \tilde{\nabla}_{\bar{j}} g_{p \bar{l}}.
\end{equation}
To handle the fourth term in \eqref{evoQ4}, we again work in product normal coordinates for $g_M$ and $g_E$. Using the same argument to control the curvature as in Lemma \ref{lemmaUnifEquiv} and the fact that $g$ and $\tilde{g}_t$ are uniformly equivalent,
\begin{equation}\label{eqn10}
 \left\vert g^{ \bar{l}k}{ R(\tilde{g}_t)_{k\bar{l}}}^{\bar{j}i}g_{i\bar{j}} \right\vert \leq C'' (\tr_{\omega} \tilde{\omega}_t ) ( \tr_{\tilde{\omega}_t} \omega) \leq C''.
\end{equation}
We combine \eqref{evoQ4}, \eqref{eqn10} and again use the uniform equivalence of $g$ and $\tilde{g}_t$, giving
\begin{eqnarray}\label{evoQ42}
\left(\frac{\partial}{\partial t} - \Delta \right) Q_4 & \leq & C' + C S - \vert \nabla \Psi \vert^2 - \vert \bar{\nabla} \Psi \vert^2 + AC'' - \frac{A}{C'''} S \nonumber \\
& \leq & - S - \vert \bar{\nabla} \Psi \vert^2 + C
\end{eqnarray}
where on the last line we choose $A$ large enough so that $C - A/C''' \leq -1$ and throw away the term $\vert \nabla \Psi \vert^2$. Also ignoring the term $\vert \bar{\nabla} \Psi \vert^2$ gives an upper bound for $Q_4$ by the maximum principle. Using Lemma \ref{lemmaUnifEquiv} then shows that $S$ is bounded above as well. Since $g \leq C\tilde{g}_0$ we also have an upper bound for $\vert \tilde{\nabla} g \vert^2_{\tilde{g}_0}$.

Now we derive \eqref{ineqS}. Notice that by definition $\vert \bar{\nabla} \Psi \vert^2 = \vert \Rm (g) - \Rm (\tilde{g}_0) \vert^2$ where we use $\Rm(\tilde{g}_0)$ for the Riemann curvature tensor of $\tilde{g}_0$, ${R(\tilde{g}_0)_{i\bar{j}k}}^l$. By Lemma \ref{lemmaTildeCurv},
\begin{equation}\label{eqn11}
\vert \Rm (g) \vert^2  \leq  2 \vert \Rm (g) - \Rm (\tilde{g}_0) \vert^2 +  2 \vert \Rm (\tilde{g}_0) \vert^2 \leq 2 \vert \bar{\nabla} \Psi \vert^2 + C.
\end{equation}
Substituting \eqref{eqn11} into \eqref{evoQ42} along with the bound on $S$ gives \eqref{ineqS}.

\end{proof}

Following \cite{SW3}, we bound the curvature tensor of $g$.

\begin{lemma}
There exists a uniform $C > 0$ such that on $X \times [0, \infty)$,
\begin{equation}
\vert \Rm(g) \vert^2 \leq C.
\end{equation}
\end{lemma}

\begin{proof}
We have the following evolution equation for curvature along the K\"ahler-Ricci flow (see \cite{SW3}):
\begin{equation}\label{evolutionCurvature}
\left( \frac{\partial}{\partial t} - \Delta \right) \vert \Rm (g) \vert \leq \frac{C_0}{2} \vert \Rm(g) \vert^2 - \frac{1}{2} \vert \Rm(g) \vert.
\end{equation}
Define the quantity $Q = \vert \Rm (g) \vert + (C_0 + 1) S$. Then using \eqref{ineqS}, \eqref{evolutionCurvature} and the maximum principle, we have the estimate
\begin{equation}
\left( \frac{\partial}{\partial t} - \Delta \right)  Q \leq - \frac{1}{2}\vert \Rm (g) \vert^2 + C,
\end{equation}
obtaining a bound for $\vert \Rm (g) \vert^2$.
\end{proof}

Using Shi's derivative estimates, we obtain bounds for the derivatives of curvature. For a proof of the following lemma, please see \cite{Shi} (or \cite{SW3} Theorem 2.15).

\begin{lemma}\label{lemmaShiDeriv}
There exists uniform $C(k)$ for $k = 0, 1, 2, \ldots$ such that on $X \times [0, \infty)$, 
\begin{equation}
|\nabla^k_{\mathbb{R}} \operatorname{Rm}(g) |^2 \leq C(k),
\end{equation}
where $\nabla_{\mathbb{R}}$ is the covariant derivative with respect to $g$ as a Riemannian metric.
\end{lemma}

\section{Higher order estimates for the metric $\omega(t)$}

We will now use the curvature bounds and the maximum principle to obtain higher order estimates for $g$. Examples of higher order estimates using similar quantities and the maximum principle can be found in \cite{Ch, DH, LSY}. 

\begin{lemma}\label{lemmaHigherOrder}
There exists uniform $C(k) > 0$ for $k = 0, 1, 2,\ldots$ such that on $X \times [0, \infty)$,
\begin{equation}
| \tilde{\nabla}^k g |^2 \leq C(k).
\end{equation}
\end{lemma}

\begin{proof}

We observe that a uniform bound on $\vert \tilde{\nabla} \Psi \vert^2$ will give a uniform bound on $\vert \tilde{\nabla} \tilde{\nabla} g \vert^2$.  We begin by calculating
\begin{eqnarray}\label{timeGradPsi}
\frac{\partial}{\partial t} \vert \tilde{\nabla} \Psi \vert^2 & = & \frac{\partial}{\partial t} \left( g^{\bar{s}r} g^{\bar{j}i} g^{\bar{l}k} g_{p\bar{q}} \tilde{\nabla}_r \Psi^p_{ik} \overline{\tilde{\nabla}_s \Psi^q_{jl}} \right) \nonumber \\
& = & -\left( -R^{\bar{s}r} - g^{\bar{s}r}\right) g^{\bar{j}i} g^{\bar{l}k} g_{p\bar{q}} \tilde{\nabla}_r \Psi^p_{ik} \overline{\tilde{\nabla}_s \Psi^q_{jl}} - g^{\bar{s}r} \left( -R^{\bar{j}i} - g^{\bar{j}i} \right)  g^{\bar{l}k} g_{p\bar{q}} \tilde{\nabla}_r \Psi^p_{ik} \overline{\tilde{\nabla}_s \Psi^q_{jl}} \nonumber \\ 
& & \ \ -  g^{\bar{s}r} g^{\bar{j}i} \left( -R^{\bar{l}k} - g^{\bar{l}k} \right) g_{p\bar{q}} \tilde{\nabla}_r \Psi^p_{ik} \overline{\tilde{\nabla}_s \Psi^q_{jl}} +  g^{\bar{s}r} g^{\bar{j}i} g^{\bar{l}k} \left( -R_{p\bar{q}} -g_{p\bar{q}} \right) \tilde{\nabla}_r \Psi^p_{ik} \overline{\tilde{\nabla}_s \Psi^q_{jl}} \nonumber \\
& & \ \ + 2 \operatorname{Re} \left( g^{\bar{s}r} g^{\bar{j}i} g^{\bar{l}k} g_{p\bar{q}} \tilde{\nabla}_r \left( \Delta \Psi^p_{ik} - {\nabla}^{\bar{b}} {R(\tilde{g}_0)_{i\bar{b}k}}^p \right)  \overline{\tilde{\nabla}_s \Psi^q_{jl}}\right).
\end{eqnarray}
Applying the Laplacian to  $\vert \tilde{\nabla} \Psi \vert^2$,
\begin{eqnarray}\label{laplaceGradPsi}
\Delta  \vert \tilde{\nabla} \Psi \vert^2 & = & \vert \nabla \tilde{\nabla} \Psi \vert^2 + \vert \bar{\nabla} \tilde{\nabla} \Psi \vert^2 +  g^{\bar{s}r} g^{\bar{j}i} g^{\bar{l}k} g_{p\bar{q}} \left( \left(\Delta \tilde{\nabla}_r \Psi^p_{ik} \right) \overline{\tilde{\nabla}_s \Psi^q_{jl}} + \tilde{\nabla}_r \Psi^p_{ik} \overline{\left(\bar{\Delta} \tilde{\nabla}_s \Psi^q_{jl}\right)}\right) \nonumber \\
& = & \vert \nabla \tilde{\nabla} \Psi \vert^2 + \vert \bar{\nabla} \tilde{\nabla} \Psi \vert^2 +  2 \operatorname{Re} \left( g^{\bar{s}r} g^{\bar{j}i} g^{\bar{l}k} g_{p\bar{q}} \left(\Delta \tilde{\nabla}_r \Psi^p_{ik} \right) \overline{\tilde{\nabla}_s \Psi^q_{jl}} \right) \nonumber \\
& & \ \ + R^{\bar{s}r} g^{\bar{j}i} g^{\bar{l}k} g_{p\bar{q}} \tilde{\nabla}_r \Psi^p_{ik} \overline{\tilde{\nabla}_s \Psi^q_{jl}} + g^{\bar{s}r} R^{\bar{j}i}  g^{\bar{l}k} g_{p\bar{q}} \tilde{\nabla}_r \Psi^p_{ik} \overline{\tilde{\nabla}_s \Psi^q_{jl}} \nonumber \\
& & \ \ + g^{\bar{s}r} g^{\bar{j}i} R^{\bar{l}k}  g_{p\bar{q}} \tilde{\nabla}_r \Psi^p_{ik} \overline{\tilde{\nabla}_s \Psi^q_{jl}} -  g^{\bar{s}r} g^{\bar{j}i} g^{\bar{l}k} R_{p\bar{q}} \tilde{\nabla}_r \Psi^p_{ik} \overline{\tilde{\nabla}_s \Psi^q_{jl}},
\end{eqnarray}
where on the last line we use a commutation formula similar to \eqref{commuteLaplace1}. Putting together \eqref{timeGradPsi} and \eqref{laplaceGradPsi}, we obtain the evolution equation
\begin{eqnarray}\label{evoGradPsi}
\left(\frac{\partial}{\partial t} - \Delta \right) \vert \tilde{\nabla}\Psi \vert^2 & = & 2 \vert \tilde{\nabla}\Psi \vert^2 -  \vert \nabla \tilde{\nabla} \Psi \vert^2 - \vert \bar{\nabla} \tilde{\nabla} \Psi \vert^2 - 2 \operatorname{Re} \left( g^{\bar{s}r} g^{\bar{j}i} g^{\bar{l}k} g_{p\bar{q}} \tilde{\nabla}_r {\nabla}^{\bar{b}} {R(\tilde{g}_0)_{i\bar{b}k}}^p \overline{\tilde{\nabla}_s \Psi^q_{jl}}\right) \nonumber \\
& & \ \ + 2 \operatorname{Re} \left( g^{\bar{s}r} g^{\bar{j}i} g^{\bar{l}k} g_{p\bar{q}} \left( \tilde{\nabla}_r \Delta - \Delta \tilde{\nabla}_r \right) \Psi_{ik}^p \overline{ \tilde{\nabla}_s \Psi_{jl}^q} \right).
\end{eqnarray}
Choose coordinates so that $\tilde{g}_0$ is the identity and $\partial_{i} \tilde{g}_0 = 0$ and $\partial_{i_1} \partial_{ i_2} \tilde{g}_0 = 0$ at a point as in \cite{T}. To deal with the fourth term in \eqref{evoGradPsi}, we calculate
\begin{eqnarray}\label{fourthTermPart}
\tilde{\nabla}_r \nabla^{\bar{b}} {R(\tilde{g}_0)_{i\bar{b}k}}^p & = & \tilde{\nabla}_r g^{\bar{b}a} \Big( \tilde{\nabla}_a {R(\tilde{g}_0)_{i\bar{b}k}}^p - \Psi_{ia}^\alpha {R(\tilde{g}_0)_{\alpha \bar{b} k}}^p - \Psi_{ka}^\alpha {R(\tilde{g}_0)_{i \bar{b} \alpha}}^p \nonumber \\
& & \ + \Psi_{\alpha a}^p {R(\tilde{g}_0)_{i \bar{b} k}}^\alpha \Big) + g^{\bar{b}a} \Big( \tilde{\nabla}_r \tilde{\nabla}_a {R(\tilde{g}_0)_{i\bar{b}k}}^p - \tilde{\nabla}_r \Psi_{ia}^\alpha {R(\tilde{g}_0)_{\alpha \bar{b} k}}^p \nonumber \\
& & \ - \Psi_{ia}^\alpha \tilde{\nabla}_r {R(\tilde{g}_0)_{\alpha \bar{b} k}}^p - \tilde{\nabla}_r \Psi_{ka}^\alpha {R(\tilde{g}_0)_{i \bar{b} \alpha}}^p - \Psi_{ka}^\alpha \tilde{\nabla}_r {R(\tilde{g}_0)_{i \bar{b} \alpha}}^p \nonumber \\
& & \  + \tilde{\nabla}_r \Psi_{\alpha a}^p {R(\tilde{g}_0)_{i\bar{b}k}}^\alpha + \Psi_{\alpha a}^p \tilde{\nabla}_r {R(\tilde{g}_0)_{i\bar{b}k}}^\alpha \Big).
\end{eqnarray}
We now bound all of the terms arising from \eqref{fourthTermPart} using Lemmas \ref{lemmaTildeCurv} and \ref{lemmaGradEst}. For the first term in \eqref{fourthTermPart},
\begin{eqnarray}\label{fourthTermP1}
\left \vert 2 \operatorname{Re} \Big( g^{\bar{s}r} g^{\bar{j}i} g^{\bar{l}k} g_{p\bar{q}}  \tilde{\nabla}_r g^{\bar{b}a}  \tilde{\nabla}_a {R(\tilde{g}_0)_{i\bar{b}k}}^p \overline{\tilde{\nabla}_s \Psi^q_{jl}}\Big) \right \vert & \leq & C \vert \tilde{\nabla} g \vert \vert \tilde{\nabla} \Rm(\tilde{g}_0) \vert \vert \tilde{\nabla} \Psi \vert \nonumber \\
& \leq & C \vert \tilde{\nabla} \Psi \vert^2 + C.
\end{eqnarray}
We bound the second, and similarly the third and fourth terms in \eqref{fourthTermPart}:
\begin{eqnarray}\label{fourthTermP2}
\left \vert 2 \operatorname{Re} \Big( g^{\bar{s}r} g^{\bar{j}i} g^{\bar{l}k} g_{p\bar{q}}  \tilde{\nabla}_r g^{\bar{b}a} \Psi_{ia}^\alpha {R(\tilde{g}_0)_{\alpha \bar{b} k}}^p \overline{\tilde{\nabla}_s \Psi^q_{jl}}\Big) \right \vert & \leq & C \vert \tilde{\nabla} g \vert \vert \Rm(\tilde{g}_0) \vert \vert \tilde{\nabla} \Psi \vert \nonumber \\
& \leq & C \vert  \tilde{\nabla} \Psi \vert^2 + C
\end{eqnarray}
Calculating similarly for the remaining terms in \eqref{fourthTermPart}, we obtain the following bound for the fourth term of \eqref{evoGradPsi}: 
\begin{equation}\label{boundForFourthTerm}
2 \operatorname{Re} \left( g^{\bar{s}r} g^{\bar{j}i} g^{\bar{l}k} g_{p\bar{q}} \tilde{\nabla}_r {\nabla}^{\bar{b}} {R(\tilde{g}_0)_{i\bar{b}k}}^p \overline{\tilde{\nabla}_s \Psi^q_{jl}}\right)  \leq  C \vert \tilde{\nabla} \Psi \vert^2 + C.
\end{equation}
Using the same coordinates as above, we compute the commutation relation for $\left( \tilde{\nabla}_r \Delta - \Delta \tilde{\nabla}_r \right) \Psi_{ik}^p$ to handle the last term in \eqref{evoGradPsi},
\begin{eqnarray}\label{huge1}
\tilde{\nabla}_r \Delta \Psi_{ik}^p & = & \tilde{\nabla}_r \Big( g^{\bar{b}a} \nabla_a \nabla_{\bar{b}} \Psi_{ik}^p \Big) \nonumber \\
& = & \partial_r g^{\bar{b}a} \Big( \partial_a \partial_{\bar{b}} \Psi_{ik}^p - \Gamma_{ia}^\alpha \partial_{\bar{b}} \Psi_{\alpha k}^p - \Gamma_{ka}^\alpha \partial_{\bar{b}} \Psi_{i\alpha}^p + \Gamma_{\alpha a}^p \partial_{\bar{b}} \Psi_{ik}^\alpha \Big) \nonumber \\
& & \ + g^{\bar{b}a} \Big( \partial_r \partial_a \partial_{\bar{b}} \Psi_{ik}^p - \partial_r \Gamma_{ia}^\alpha \partial_{\bar{b}} \Psi_{\alpha k}^p - \Gamma_{ia}^\alpha \partial_r \partial_{\bar{b}} \Psi_{\alpha k}^p - \partial_r \Gamma_{ka}^\alpha \partial_{\bar{b}} \Psi_{i \alpha}^p \nonumber \\
& & \ \ \ - \Gamma_{k a}^\alpha \partial_r \partial_{\bar{b}} \Psi_{i \alpha}^p + \partial_r \Gamma_{\alpha a}^p \partial_{\bar{b}} \Psi_{ik}^\alpha + \Gamma_{\alpha a}^p \partial_r \partial_{\bar{b}} \Psi_{ik}^\alpha \Big).
\end{eqnarray}

\begin{eqnarray}\label{huge2}
\Delta \tilde{\nabla}_r \Psi_{ik}^p & = & g^{\bar{b}a} \nabla_a \nabla_{\bar{b}} \tilde{\nabla}_r \Psi_{ik}^p \nonumber \\
& = & g^{\bar{b}a} \Big( \partial_r \partial_a \partial_{\bar{b}} \Psi_{ik}^p - \Gamma_{ra}^\beta \partial_\beta \partial_{\bar{b}} \Psi_{ik}^p - \Gamma_{ia}^\beta \partial_r \partial_{\bar{b}} \Psi_{\beta k}^p - \Gamma_{ka}^\beta \partial_r \partial_{\bar{b}} \Psi_{i\beta}^p + \Gamma_{\beta a}^p \partial_r \partial_{\bar{b}} \Psi_{ik}^\beta\nonumber \\
& & \  - \partial_a {R(\tilde{g}_0)_{i\bar{b}r}}^\alpha \Psi_{\alpha k}^p - {R(\tilde{g}_0)_{i\bar{b}r}}^\alpha \partial_a \Psi_{\alpha k}^p + \Gamma_{i a}^\beta {R(\tilde{g}_0)_{\beta \bar{b} r}}^\alpha \Psi_{\alpha k}^p \nonumber \\
& & \ + \Gamma_{ra}^\beta {R(\tilde{g}_0)_{i \bar{b} \beta}}^\alpha \Psi_{\alpha k}^p - \Gamma_{\beta a}^\alpha {R(\tilde{g}_0)_{i \bar{b} r}}^\beta \Psi_{\alpha k}^p + \Gamma_{\alpha a}^\beta {R(\tilde{g}_0)_{i \bar{b} r}}^\alpha \Psi_{\beta k}^p \nonumber \\
& & \ + \Gamma_{ka}^\beta {R(\tilde{g}_0)_{i \bar{b}r}}^\alpha \Psi_{\alpha \beta}^p - \Gamma_{\beta a}^p {R(\tilde{g}_0)_{i \bar{b} r}}^\alpha \Psi_{\alpha k}^\beta - \partial_a {R(\tilde{g}_0)_{k \bar{b} r}}^\alpha \Psi_{i \alpha}^p \nonumber \\
& & \ - {R(\tilde{g}_0)_{k \bar{b} r}}^\alpha \partial_a \Psi_{i \alpha}^p + \Gamma_{ka}^\beta {R(\tilde{g}_0)_{\beta \bar{b} r}}^\alpha \Psi_{i\alpha}^p + \Gamma_{ra}^\beta {R(\tilde{g}_0)_{k \bar{b} \beta}}^\alpha \Psi_{i\alpha}^p \nonumber \\
& & - \Gamma_{\beta a}^\alpha {R(\tilde{g}_0)_{k \bar{b} r}}^\beta \Psi_{i \alpha}^p + \Gamma_{ia}^\beta {R(\tilde{g}_0)_{k \bar{b} r}}^\alpha \Psi_{\beta \alpha}^p + \Gamma_{\alpha a}^\beta {R(\tilde{g}_0)_{k \bar{b}r}}^p \Psi_{i \beta}^p \nonumber \\
& & - \Gamma_{\beta a}^p {R(\tilde{g}_0)_{k \bar{b} r}}^\alpha \Psi_{i \alpha}^\beta + \partial_a {R(\tilde{g}_0)_{\alpha \bar{b} r}}^p \Psi_{ik}^\alpha + {R(\tilde{g}_0)_{\alpha \bar{b} r}}^p \partial_a \Psi_{ik}^\alpha \nonumber \\
& & - \Gamma_{\alpha a}^\beta {R(\tilde{g}_0)_{\beta \bar{b} r}}^p \Psi_{ik}^\alpha - \Gamma_{ra}^\beta {R(\tilde{g}_0)_{\alpha \bar{b} \beta}}^p \Psi_{ik}^\alpha + \Gamma_{\beta a}^p {R(\tilde{g}_0)_{\alpha \bar{b} r}}^\beta \Psi_{ik}^p \nonumber \\
& & - \Gamma_{ia}^\beta {R(\tilde{g}_0)_{\alpha \bar{b} r}}^p \Psi_{\beta k}^\alpha - \Gamma_{ka}^\beta {R(\tilde{g}_0)_{\alpha \bar{b} r}}^p \Psi_{i\beta}^\alpha + \Gamma_{\beta a}^\alpha {R(\tilde{g}_0)_{\alpha \bar{b} r}}^p \Psi_{ik}^\beta \Big).
\end{eqnarray}
Putting these together and making use of our choice of coordinates,
\begin{eqnarray}\label{huge4}
\Big( \tilde{\nabla}_r \Delta - \Delta \tilde{\nabla}_r \Big) \Psi_{ik}^p & = & \tilde{\nabla}_r g^{\bar{b}a} \Big( \tilde{\nabla}_a ( {R_{i\bar{b}k}}^p - {R(\tilde{g}_0)_{i \bar{b}k}}^p ) - \Psi_{ia}^\alpha ({R_{\alpha \bar{b} k}}^p - {R(\tilde{g}_0)_{\alpha \bar{b} k}}^p ) \nonumber \\
& & \ - \Psi_{ka}^\alpha ( {R_{i \bar{b} \alpha}}^p - {R(\tilde{g}_0)_{i \bar{b} \alpha}}^p ) + \Psi_{\alpha a}^p ( {R_{i \bar{b} k}}^\alpha - {R(\tilde{g}_0)_{i \bar{b} k}}^\alpha ) \Big) \nonumber \\
& & + g^{\bar{b}a} \Big( -\tilde{\nabla}_r \Psi_{ia}^\alpha  ({R_{\alpha \bar{b} k}}^p - {R(\tilde{g}_0)_{\alpha \bar{b} k}}^p ) - \tilde{\nabla}_r \Psi_{ka}^\alpha ( {R_{i \bar{b} \alpha}}^p - {R(\tilde{g}_0)_{i \bar{b} \alpha}}^p )  \nonumber \\
& & \ + \tilde{\nabla}_r \Psi_{\alpha a}^p ( {R_{i \bar{b} k}}^\alpha - {R(\tilde{g}_0)_{i \bar{b} k}}^\alpha ) + \Psi_{ra}^\beta \tilde{\nabla}_\beta ( {R_{i \bar{b} k}}^p - {R(\tilde{g}_0)_{i\bar{b}k}}^p ) \nonumber \\
& & \  + \tilde{\nabla}_a {R(\tilde{g}_0)_{i\bar{b}r}}^\alpha \Psi_{\alpha k}^p + {R(\tilde{g}_0)_{i\bar{b}r}}^\alpha \tilde{\nabla}_a \Psi_{\alpha k}^p - \Psi_{i a}^\beta {R(\tilde{g}_0)_{\beta \bar{b} r}}^\alpha \Psi_{\alpha k}^p \nonumber \\
& & \ - \Psi_{ra}^\beta {R(\tilde{g}_0)_{i \bar{b} \beta}}^\alpha \Psi_{\alpha k}^p + \Psi_{\beta a}^\alpha {R(\tilde{g}_0)_{i \bar{b} r}}^\beta \Psi_{\alpha k}^p - \Psi_{\alpha a}^\beta {R(\tilde{g}_0)_{i \bar{b} r}}^\alpha \Psi_{\beta k}^p \nonumber \\
& & \ - \Psi_{ka}^\beta {R(\tilde{g}_0)_{i \bar{b}r}}^\alpha \Psi_{\alpha \beta}^p + \Psi_{\beta a}^p {R(\tilde{g}_0)_{i \bar{b} r}}^\alpha \Psi_{\alpha k}^\beta + \tilde{\nabla}_a {R(\tilde{g}_0)_{k \bar{b} r}}^\alpha \Psi_{i \alpha}^p \nonumber \\
& & \ + {R(\tilde{g}_0)_{k \bar{b} r}}^\alpha \tilde{\nabla}_a \Psi_{i \alpha}^p - \Psi_{ka}^\beta {R(\tilde{g}_0)_{\beta \bar{b} r}}^\alpha \Psi_{i\alpha}^p - \Psi_{ra}^\beta {R(\tilde{g}_0)_{k \bar{b} \beta}}^\alpha \Psi_{i\alpha}^p \nonumber \\
& & + \Psi_{\beta a}^\alpha {R(\tilde{g}_0)_{k \bar{b} r}}^\beta \Psi_{i \alpha}^p - \Psi_{ia}^\beta {R(\tilde{g}_0)_{k \bar{b} r}}^\alpha \Psi_{\beta \alpha}^p - \Psi_{\alpha a}^\beta {R(\tilde{g}_0)_{k \bar{b}r}}^p \Psi_{i \beta}^p \nonumber \\
& & + \Psi_{\beta a}^p {R(\tilde{g}_0)_{k \bar{b} r}}^\alpha \Psi_{i \alpha}^\beta - \tilde{\nabla}_a {R(\tilde{g}_0)_{\alpha \bar{b} r}}^p \Psi_{ik}^\alpha - {R(\tilde{g}_0)_{\alpha \bar{b} r}}^p \tilde{\nabla}_a \Psi_{ik}^\alpha \nonumber \\
& & + \Psi_{\alpha a}^\beta {R(\tilde{g}_0)_{\beta \bar{b} r}}^p \Psi_{ik}^\alpha + \Psi_{ra}^\beta {R(\tilde{g}_0)_{\alpha \bar{b} \beta}}^p \Psi_{ik}^\alpha - \Psi_{\beta a}^p {R(\tilde{g}_0)_{\alpha \bar{b} r}}^\beta \Psi_{ik}^p \nonumber \\
& & + \Psi_{ia}^\beta {R(\tilde{g}_0)_{\alpha \bar{b} r}}^p \Psi_{\beta k}^\alpha + \Psi_{ka}^\beta {R(\tilde{g}_0)_{\alpha \bar{b} r}}^p \Psi_{i\beta}^\alpha - \Psi_{\beta a}^\alpha {R(\tilde{g}_0)_{\alpha \bar{b} r}}^p \Psi_{ik}^\beta \Big).
\end{eqnarray}
Using \eqref{huge4} and Lemmas \ref{lemmaTildeCurv}, \ref{lemmaGradEst} and \ref{lemmaShiDeriv}, we can bound all the terms resulting from the final term of \eqref{evoGradPsi}. Starting with the first term from \eqref{huge4}:
\begin{equation}\label{finalTermP1}
2 \operatorname{Re} \left( g^{\bar{s}r} g^{\bar{j}i} g^{\bar{l}k} g_{p\bar{q}} \tilde{\nabla}_r g^{\bar{b}a} \tilde{\nabla}_a {R_{i\bar{b}k}}^p       \overline{ \tilde{\nabla}_s \Psi_{jl}^q} \right) \leq C \vert \tilde{\nabla} g \vert \vert \tilde{\nabla} \Rm(g) \vert \vert \tilde{\nabla} \Psi \vert.
\end{equation}
We bound $\vert \tilde{\nabla} \Rm(g) \vert$ by observing that
\begin{equation}
\left( \tilde{\nabla}_a - \nabla_a \right) {R_{i\bar{l}p}}^r = \Psi_{ia}^\alpha {R_{\alpha \bar{l} p}}^r + \Psi_{pa}^\alpha {R_{i \bar{l} \alpha}}^r - \Psi_{\alpha a}^r {R_{i \bar{l} p}}^\alpha,
\end{equation}
and so
\begin{equation}\label{finalTermP2}
\vert \tilde{\nabla} \Rm(g) \vert^2 \leq 2 \vert (\tilde{\nabla} - \nabla) \Rm(g) \vert^2 + 2 \vert \nabla \Rm(g) \vert^2 \leq C \vert \Psi \vert^2 + C\vert \Rm(g) \vert^2 + 2 \vert \nabla \Rm(g) \vert^2 \leq C
\end{equation}
where to get the last inequality we use Lemmas \ref{lemmaGradEst} and \ref{lemmaShiDeriv}. Substituting \eqref{finalTermP2} into \eqref{finalTermP1} gives the bound
\begin{equation}\label{finalTermP3}
2 \operatorname{Re} \left( g^{\bar{s}r} g^{\bar{j}i} g^{\bar{l}k} g_{p\bar{q}} \tilde{\nabla}_r g^{\bar{b}a} \tilde{\nabla}_a {R_{i\bar{b}k}}^p       \overline{ \tilde{\nabla}_s \Psi_{jl}^q} \right) \leq C \vert \tilde{\nabla} \Psi \vert \leq C \vert \tilde{\nabla} \Psi \vert^2 + C
\end{equation}
For the second term from \eqref{huge4}, using Lemmas \ref{lemmaTildeCurv} and \ref{lemmaGradEst}, 
\begin{eqnarray}
2 \operatorname{Re} \left( g^{\bar{s}r} g^{\bar{j}i} g^{\bar{l}k} g_{p\bar{q}} \tilde{\nabla}_r g^{\bar{b}a} {R(\tilde{g}_0)_{i\bar{b}k}}^p \overline{ \tilde{\nabla}_s \Psi_{jl}^q} \right) & \leq & C \vert \tilde{\nabla} g \vert \vert \tilde{\nabla} \Rm(\tilde{g}_0) \vert  \vert \tilde{\nabla} \Psi \vert \nonumber \\
& \leq & C \vert \tilde{\nabla} \Psi \vert^2 + C.
\end{eqnarray}
Similarly, we bound the remaining terms arising from \eqref{huge4} and obtain the estimate
\begin{equation}\label{huge3}
\vert 2 \operatorname{Re} \Big( g^{\bar{s}r} g^{\bar{j}i} g^{\bar{l}k} g_{p\bar{q}} \Big( \tilde{\nabla}_r \Delta - \Delta \tilde{\nabla}_r \Big) \Psi_{ik}^p \overline{ \tilde{\nabla}_s \Psi_{jl}^q} \Big)\vert  \leq C \vert \tilde{\nabla} \Psi \vert^2 + C.
\end{equation}
Substituting \eqref{boundForFourthTerm} and \eqref{huge3} into \eqref{evoGradPsi},
\begin{equation}\label{evoGradPsi2}
\left( \frac{\partial}{\partial t} - \Delta \right)  \vert \tilde{\nabla}\Psi \vert^2 \leq C_2 \vert \tilde{\nabla}\Psi \vert^2 + C.
\end{equation}
By the definition of $\Psi$,
\begin{equation}
\nabla_l \Psi_{ij}^k - \tilde{\nabla}_l \Psi_{ij}^k = - \Psi_{li}^\alpha \Psi_{\alpha j}^k - \Psi_{lj}^\alpha \Psi_{i \alpha}^k + \Psi_{l \alpha}^k \Psi_{ij}^\alpha.
\end{equation}
Using this with the Lemma \ref{lemmaGradEst}, we have
\begin{equation}\label{equivPsi}
\vert \tilde{\nabla} \Psi \vert^2 \leq 2 \vert \nabla \Psi \vert^2 +  2 \vert \tilde{\nabla} \Psi - \nabla \Psi \vert^2 \leq 2 \vert \nabla \Psi \vert^2 + C .
\end{equation}
Define the quantity $Q_1 = \vert \tilde{\nabla} \Psi \vert^2 + 2(C_1 + 1) \vert \Psi \vert^2$. Then using \eqref{evoS2}, \eqref{evoGradPsi2}, \eqref{equivPsi} and Lemma \ref{lemmaGradEst},
\begin{eqnarray}
\left( \frac{\partial}{\partial t} - \Delta \right) Q_1 & \leq &  C_1 \vert \tilde{\nabla} \Psi \vert^2 + C + 2(C_1 + 1) \left( C + C\vert \Psi \vert^2 - \vert \nabla \Psi \vert^2 - \vert \bar{\nabla} \Psi \vert^2 \right) \nonumber \\
& \leq & - \vert \tilde{\nabla} \Psi \vert^2 + C.
\end{eqnarray}
This gives a uniform bound for  $\vert \tilde{\nabla} \Psi \vert^2$ and hence a uniform bound for  $\vert \tilde{\nabla} g \vert^2$.

Now we may proceed inductively to derive estimates of any order. As in the case when $k=1$, it will suffice to bound $\vert \tilde{\nabla}^k \Psi \vert^2$ by induction. Computing as in \eqref{evoGradPsi}, the evolution equation of $\vert \tilde{\nabla}^k \Psi \vert^2$ is
\begin{eqnarray}\label{evoGradKPsi}
\left(\frac{\partial}{\partial t} - \Delta \right) \vert \tilde{\nabla}^k \Psi \vert^2 & = & (k+1) \vert \tilde{\nabla}^k\Psi \vert^2 -  \vert \nabla \tilde{\nabla}^k \Psi \vert^2 - \vert \bar{\nabla} \tilde{\nabla}^k \Psi \vert^2 - 2 \operatorname{Re}\left< \tilde{\nabla}^k T , \tilde{\nabla}^k \Psi \right> \nonumber \\
& & \ \ + 2 \operatorname{Re} \left< \left(\tilde{\nabla}^k \Delta - \Delta \tilde{\nabla}^k \right) \Psi , \tilde{\nabla}^k \Psi \right>,
\end{eqnarray} 
where $\left< \cdot , \cdot \right>$ denotes the inner product with respect to $g$ and where $T$ is the tensor $T_{ij}^k = \nabla^{\bar{b}} {R_{i\bar{b}j}}^k$. We work in coordinates where $\tilde{g}_0$ is the identity and $\partial_i \tilde{g}_0 = 0, \partial_{i_1} \partial_{i_2} \tilde{g}_0 = 0, \ldots , \partial_{i_1} \partial_{i_2} \ldots \partial_{i_{k+1}}\tilde{g}_0 = 0$ at a point as in \cite{T}. Using these coordinates, $\tilde{\Gamma} = 0$, \ldots, $\tilde{\nabla}^k \tilde{\Gamma} = 0$ and $\Gamma = \Psi, \ldots, \tilde{\nabla}^k \Gamma = \tilde{\nabla}^k \Psi$. Proceeding as we did to obtain \eqref{boundForFourthTerm}, we bound the fourth term in \eqref{evoGradKPsi} by $C \vert \tilde{\nabla}^k \Psi \vert^2 + C$ since all lower order derivatives of $\Psi$ are bounded by induction. As in \eqref{huge4}, the final term is made up of terms involving derivatives of curvature tensors and derivatives of $\Psi$ of order less than or equal to $k$. All terms here are good, since a $k$-th order derivative of $\Psi$ is what we are estimating, and by induction lower order derivatives of $\Psi$ are bounded. Derivatives of order less than or equal to $k$ of $\Rm(g)$ are bounded by induction and Lemma \ref{lemmaShiDeriv} since differentiation with respect to $g$ and $\tilde{g}_0$ differ by terms involving lower order derivatives of $\Psi$ as in \eqref{finalTermP2}. Any derivatives of $\Rm(\tilde{g}_0)$ are bounded by Lemma \ref{lemmaTildeCurv}. As above, we obtain the estimate
\begin{equation}\label{evoGradKPsiV2}
\left(\frac{\partial}{\partial t} - \Delta \right) \vert \tilde{\nabla}^k \Psi \vert^2 \leq C_k \vert \tilde{\nabla}^k \Psi \vert^2 + C.
\end{equation}
We define the quantity $Q_k = \vert \tilde{\nabla}^k \Psi \vert^2 + 2(C_k + 1) \vert \tilde{\nabla}^{k-1} \Psi \vert^2$. We have the inequality
\begin{eqnarray}
\vert \tilde{\nabla}^k \Psi \vert^2 & \leq & 2 \vert \nabla \tilde{\nabla}^{k-1} \Psi \vert^2 + 2 \vert (\nabla - \tilde{\nabla}) \tilde{\nabla}^{k-1} \Psi \vert^2 \nonumber \\
 & \leq & 2 \vert \nabla \tilde{\nabla}^{k-1} \Psi \vert^2 + C
\end{eqnarray}
since $ (\nabla - \tilde{\nabla}) \tilde{\nabla}^{k-1} \Psi$ is made up of terms involving $\Psi$ and $\tilde{\nabla}^{k-1} \Psi$ and hence is bounded by the induction hypothesis. Then using this and \eqref{evoGradKPsiV2}, we have
\begin{eqnarray}
\left( \frac{\partial}{\partial t} - \Delta \right) Q_k & \leq & C_k \vert \tilde{\nabla}^k \Psi \vert^2 + C + 2(C_k+1) \left( C - \vert  \nabla \tilde{\nabla}^{k-1} \Psi \vert^2 \right) \nonumber \\
& \leq & - \vert \tilde{\nabla}^k \Psi \vert^2 + C
\end{eqnarray}
giving us a bound for $\vert \tilde{\nabla}^k \Psi \vert^2$.
\end{proof}

Because of the symmetries of the metric tensor $g_{i \bar{j}}$, we obtain the following lemma bounding the barred derivatives of the metric. 

\begin{lemma}\label{lemmaBarredDerivs}
There exists uniform $C(k) > 0$ for $k = 0, 1, 2,\ldots$ such that on $X \times [0, \infty)$,
\begin{equation}
\vert \bar{\tilde{\nabla}}^k g \vert^2 \leq C(k).
\end{equation}
\end{lemma}

Using Lemmas \ref{lemmaHigherOrder} and \ref{lemmaBarredDerivs}, we construct estimates for all possible covariant derivatives of the metric.

\begin{lemma}\label{lemmaHigherOrder2}
There exists uniform $C(k) > 0$  for $k = 0, 1, 2, \ldots$ such that on $X \times [0,\infty)$,
\begin{equation}
\vert \tilde{\nabla}^k_{\mathbb{R}} g \vert^2 \leq C(k),
\end{equation}
where $\tilde{\nabla}_{\mathbb{R}}$ is the covariant derivative with respect to $\tilde{g}_0$ as a Riemannian metric. 
\end{lemma}

\begin{proof}
Let $\mathbf{a} = (a_1, a_2, \ldots, a_k)$ be a $k$-tuple with symbolic entries $z$ or $\bar{z}$. We define $\tilde{\nabla}^{a_i}$ to be the operator $\tilde{\nabla}$ if $a_i = z$ or $\bar{\tilde{\nabla}}$ if $a_i = \bar{z}$. Then we define $\tilde{\nabla}^{\mathbf{a}}$ to be the operator $\tilde{\nabla}^{a_1} \ldots \tilde{\nabla}^{a_k}$ (if $\mathbf{a}$ is a $0$-tuple, define $\tilde{\nabla}^{\mathbf{a}}$ to be the identity). To prove the lemma, it suffices to bound the quantity $\vert \tilde{\nabla}^{\mathbf{a}} g \vert^2$.

We will proceed by induction on $k$. The case where $k=1$ is handled by Lemmas \ref{lemmaHigherOrder} and \ref{lemmaBarredDerivs}. For the general $k$ we may assume that there exists an index $l$ such that $a_l = z$, otherwise we are done by Lemma \ref{lemmaBarredDerivs}. Choose $l$ to be the greatest index such that $a_l = z$ and define $\mathbf{a}'$ to be the $(l-1)$-tuple containing the first $l-1$ entries of $\mathbf{a}$. If $l = k$, we observe that a bound on $\vert \tilde{\nabla}^{\mathbf{a}} g \vert^2$ will follow from a bound on $\vert \tilde{\nabla}^{\mathbf{a}'} \Psi \vert^2$.

We will introduce some notation: if $A$ and $B$ are tensors, let $A * B$ denote any linear combination of products of $A$ and $B$ formed by contractions with the metric $g$. If $l$ is not equal to $k$,  by commuting the covariant derivatives, we have 
\begin{eqnarray}
\tilde{\nabla}^{\mathbf{a}} g & = & \tilde{\nabla}^{\mathbf{a}'} \tilde{\nabla} \bar{\tilde{\nabla}}^{k-l} g \nonumber \\
& = &  \tilde{\nabla}^{\mathbf{a}'} \Big( \bar{\tilde{\nabla}} \tilde{\nabla} \bar{\tilde{\nabla}}^{k-l-1} g + \Rm(\tilde{g}_0) * \bar{\tilde{\nabla}}^{k-l-1} g \Big) \nonumber \\
& = &  \tilde{\nabla}^{\mathbf{a}'} \Big( \bar{\tilde{\nabla}}^{k-l} \tilde{\nabla} g + \bar{\tilde{\nabla}}^{k-l-1} \Rm(\tilde{g}_0) * g + \ldots + \Rm(\tilde{g}_0) *  \bar{\tilde{\nabla}}^{k-l-1} g \Big).
\end{eqnarray}
Hence a bound on $\vert \tilde{\nabla}^{\mathbf{a}} g \vert^2$ follows from a bound on $\vert \tilde{\nabla}^{\mathbf{a}'} \bar{\tilde{\nabla}}^{k-l} \Psi \vert^2$  since the other terms are bounded by Lemma \ref{lemmaTildeCurv} and induction. We will now complete the proof by bounding $\vert \tilde{\nabla}^{\mathbf{a}'} \Psi \vert^2$ for a general $(k-1)$-tuple $\mathbf{a}'$.

Notice that if every entry of $\mathbf{a}'$ is $z$ or if every entry of $\mathbf{a'}$ is $\bar{z}$, the proof is complete by Lemmas \ref{lemmaHigherOrder} and \ref{lemmaBarredDerivs}. Now let $r$ be the greatest index such that $a'_r = \bar{z}$ and define $\mathbf{a}''$ to be the $(r-1)$-tuple containing the first $r-1$ entries of $\mathbf{a}'$. If $r = k-1$, then
\begin{equation}\label{eqn1234}
\vert \tilde{\nabla}^{\mathbf{a}'} \Psi \vert^2 = \vert \tilde{\nabla}^{\mathbf{a}''} \bar{\tilde{\nabla}} \Psi \vert^2 = \vert \tilde{\nabla}^{\mathbf{a}''} (\Rm(g) - \Rm(\tilde{g}_0) )\vert^2 \leq \vert \tilde{\nabla}^{\mathbf{a}''} \Rm(g) \vert ^2 + \vert \tilde{\nabla}^{\mathbf{a}''} (\Rm(\tilde{g}_0) \vert^2.
\end{equation}
Notice that the second term in the right hand side of \eqref{eqn1234} is bounded by Lemma \ref{lemmaTildeCurv}. We observe that $\tilde{\nabla}^{\mathbf{a}''} \Rm(g)$ differs from ${\nabla}^{\mathbf{a}''} \Rm(g)$ only by terms involving $\Rm(g), \ldots, \nabla^{k-3}_\mathbb{R} \Rm(g)$ and $\Psi, \ldots, \tilde{\nabla}^{k-3}_\mathbb{R} \Psi$. By induction and Lemma \ref{lemmaShiDeriv}, we have a bound for $\tilde{\nabla}^{\mathbf{a}''} \Rm(g)$ and hence 
\begin{equation}\label{eqn12345}
\vert \tilde{\nabla}^{\mathbf{a}'} \Psi \vert^2 \leq C.
\end{equation}
If $r < l-1$, we commute the covariant derivatives,
\begin{eqnarray}\label{eqn4321}
\tilde{\nabla}^{\mathbf{a}'} \Psi & = & \tilde{\nabla}^{\mathbf{a}''} \bar{\tilde{\nabla}} \tilde{\nabla}^{l-1-r} \Psi \nonumber \\
& = & \tilde{\nabla}^{\mathbf{a}''} \Big( \tilde{\nabla} \bar{\tilde{\nabla}} \tilde{\nabla}^{l-r-2} \Psi + \Rm(\tilde{g}_0) * \tilde{\nabla}^{l-r-2} \Psi \Big) \nonumber \\
& = & \tilde{\nabla}^{\mathbf{a}''} \Big( \tilde{\nabla}^{l-r-1} \bar{\tilde{\nabla}} \Psi + \tilde{\nabla}^{l-r-2} \Rm(\tilde{g}_0) * \Psi + \ldots + \Rm(\tilde{g}_0) * \tilde{\nabla}^{l-r-2} \Psi \Big).
\end{eqnarray}
Notice that the norm of the first term of \eqref{eqn4321} is bounded as in \eqref{eqn12345} and the norms of the other terms are bounded by induction and Lemma \ref{lemmaTildeCurv}, completing the proof. 
\end{proof}

\section{Convergence}

In this section we will complete the proof of the main theorem by showing that $\omega(t)$ converges smoothly to $\omega_M$ as $t \to \infty$. Fix $z \in M$ and define a function $\rho_z$ on $E(z) := \pi^{-1}_M (z)$ by
\begin{equation}
\omega_0 |_{E(z)} + \frac{\sqrt{-1}}{2\pi} \partial \bar{\partial} \rho_z > 0, \ \ \ \Ric \left( \omega_0 |_{E(z)} + \frac{\sqrt{-1}}{2\pi} \partial \bar{\partial} \rho_z \right) = 0, \ \ \ \int_{E(z)} \rho_z \omega_0^n = 0.
\end{equation}
Note that since $\rho_z$ varies smoothly with $z$, we may define a smooth function $\rho(z,e)$ on $X$. Then
\begin{equation}
\omega_{flat} := \omega_0 + \frac{\sqrt{-1}}{2\pi} \partial \bar{\partial} \rho
\end{equation}
determines a closed $(1,1)$-form on $X$ with $[\omega_{flat}] = [\omega_0]$. Also, $\omega_{flat}$ may not be a metric on $X$, but $\omega_{flat} |_{E(z)}$ is a flat K\"ahler metric on each fiber.

We will now prove the following estimate for $\varphi$, which will give us the convergence of $\omega(t)$.

\begin{lemma}\label{lemmaPhiTo0}
There exists uniform $C > 0$ such that on $X \times [0, \infty)$,
\begin{equation}
\vert \varphi \vert \leq C(1 + t) e^{-t}.
\end{equation}
\end{lemma} 

\begin{proof}
This proof follows similarly as in \cite{SW3}. To simplify notation, let $b_k$ denote the binomial coefficient $b_k = \tbinom{m+n}{k}.$ Then using \eqref{eqnOmega} and the fact that $[\omega_{flat}] = [\omega_0]$, 
\begin{equation}
\Omega = b_m \omega_M^m \wedge \omega_{flat}^n.
\end{equation}
We define the quantity $Q = \varphi - e^{-t} \rho$ and calculate its evolution
\begin{eqnarray}
\frac{\partial}{\partial t} Q & = & \log \frac{e^{nt} \left( \hat{\omega}_t + \frac{\sqrt{-1}}{2\pi} \partial \bar{\partial} \varphi \right)^{m+n}}{b_m  \omega_M^m \wedge \omega_{flat}^n} - \varphi + e^{-t} \rho \nonumber \\
& = & \log \frac{e^{nt}\left( e^{-t} \omega_{flat} + \left( 1 - e^{-t} \right) \omega_M + \sqrt{-1} \partial \bar{\partial} Q\right)^{m+n}}{b_m  \omega_M^m \wedge \omega_{flat}^n} - Q.
\end{eqnarray}
Now let $Q_1 = e^t Q - A t$ where $A$ is a constant to be determined later. Suppose $Q_1$ attains its maximum at a point $(z_0, t_0)$ with $t_0 > 0$, then at that point
\begin{eqnarray}
0 & \leq & \frac{\partial}{\partial t} Q_1 \leq e^t \log \frac{e^{nt}\left( e^{-t} \omega_{flat} + \left( 1 - e^{-t} \right) \omega_M\right)^{m+n}}{b_m \omega_M^m \wedge \omega_{flat}^n} - A \nonumber \\
& = & e^t \log \frac{e^{nt}\left( b_m e^{-nt}(1-e^{-t})^m \omega^m_M \wedge \omega_{flat}^n + \ldots + e^{-\left(m+n\right)t} \omega_{flat}^{m+n}\right)}{b_m \omega_M^m \wedge \omega_{flat}^n} - A \nonumber \\
& \leq & e^t \log \left( 1 + C_1 e^{-t} + \ldots + C_m e^{-mt} \right) - A \nonumber \\
& \leq & C - A.
\end{eqnarray}
If we choose $A > C$, we obtain a contradiction and hence $Q_1$ must attain its maximum at $t = 0$. This gives the estimate $\varphi \leq C \left( 1 + t \right) e^{-t}$, and we can similarly obtain a lower bound.
\end{proof}

We may now complete the proof of the main theorem.

\begin{proof}
Using Lemma \ref{lemmaHigherOrder2}, Lemma \ref{lemmaPhiTo0} and the definition of $\omega(t)$, we immediately see that $\omega(t) \to \omega_M$ in $C^{\infty}$ as $t \to \infty$ proving part (a).

We will restrict Lemma \ref{lemmaGradEst} to $E(z)$ using a method similar to that in \cite{To}. Choose complex coordinates $x^{m+1}, \ldots, x^{m+n}$ on $E$ so that $g_E$ is the identity and $g|_E$ is diagonal with entries $\lambda_{m+1}, \ldots, \lambda_{m+n}$. Then choose complex coordinates $x^1, \ldots, x^m$ on $X$ such that at a point $p$ the space spanned by $\frac{\partial}{\partial x^1} \vert_p, \ldots, \frac{\partial}{\partial x^m} \vert_p$ is orthogonal to the space spanned by $\frac{\partial}{\partial x^{m+1}} \vert_p, \ldots, \frac{\partial}{\partial x^{m+n}} \vert_p$ with respect to $g$. In this coordinate system, $g$ is diagonal with entries $\lambda_1, \ldots, \lambda_{m+n}$, and so
\begin{eqnarray}\label{restrictionBound1}
\left\vert \nabla_E g|_{E(z)} \right\vert^2_{g|_{E(z)}}  & =  & \sum_{i,j,k = m+1}^{m+n} \frac{1}{\lambda_i \lambda_j \lambda_k} \tilde{\nabla}_k g_{i\bar{j}}|_{E(z)} \overline{ \tilde{\nabla}_k g_{i\bar{j}}|_{E(z)} }\nonumber \\
& \leq & \sum_{i,j,k = 1}^{m+n} \frac{1}{\lambda_i \lambda_j \lambda_k} \tilde{\nabla}_k g_{i\bar{j}} \overline {\tilde{\nabla}_k g_{i\bar{j}} }\nonumber \\
& = & \vert \tilde{\nabla} g \vert^2 \leq C.
\end{eqnarray}
By restricting the uniform equivalence of $g$ and $\tilde{g}_t$ to $E(z)$, we see that $g|_{E(z)}$ is uniformly equivalent to $e^{-t} g_E$. Using this fact coupled with \eqref{restrictionBound1} we estimate the derivative of $e^t g|_{E(z)}$.
\begin{eqnarray}\label{restrictionBound2}
\vert \nabla_E e^t g|_{E(z)} \vert^2_{g_E} &  = & e^{2t} g_E^{\bar{j}i} g_E^{\bar{l} k} g_E^{\bar{q} p} \nabla_{E,i} (g|_{E(z)})_{k\bar{q}} \overline{ \nabla_{E, j} (g|_{E(z)})_{l\bar{p}} }\nonumber \\
& \leq & C e^{-t} (g_E)^{\bar{j}i} (g_E)^{\bar{l}k} (g_E)^{\bar{q}p}  \nabla_{E,i} (g|_{E(z)})_{k\bar{q}}\overline{ \nabla_{E, j} (g|_{E(z)})_{l\bar{p}} } \nonumber \\
& = & C e^{-t} \left\vert \nabla_E g|_{E(z)} \right\vert^2_{g|_{E(z)}} \nonumber \\
& \leq & C' e^{-t}.
\end{eqnarray}
Similarly, we obtain estimates for the $k$-th order derivative of $e^t g|_{E(z)}$:
\begin{equation}\label{restrictionBoundk}
\vert \nabla^k_E e^t g|_{E(z)} \vert^2_{g_E} \leq C e^{-kt}.
\end{equation}
We constructed $g_{flat}$ to be a flat metric when restricted to the complex torus $E(z)$, and so it is given by a constant Hermitian metric on $\mathbb{C}^n$. Using a standard coordinate system for $E(z)$, we see that $\nabla^k_E g_{flat} = 0$ for all $k$, thus
\begin{equation}\label{convBoundk}
\vert \nabla^k_E ( e^t g|_{E(z)} - g_{flat}|_{E(z)} ) \vert^2_{g_E} \leq C e^{-kt}.
\end{equation}

It remains to show that $e^t g|_{E(z)} \to g_{flat}|_{E(z)}$ in $C^0(E(z))$. Define a function $\psi$ on $E(z)$ by
\begin{equation}
\psi = e^{-t} \varphi|_{E(z)} - \rho_z.
\end{equation}
Letting $\Delta_E$ denote the Laplacian with respect to $g_E$,
\begin{equation}\label{laplaceFiberPsi}
\Delta_E \psi = \tr_{g_E} ( e^t g|_{E(z)} - g_{flat}|_{E(z)}).
\end{equation}
Combining \eqref{convBoundk} with $k = 1$ and \eqref{laplaceFiberPsi} gives the estimate
\begin{equation}\label{psiDecay}
\vert \nabla_E \Delta_E \psi \vert^2_{g_E} \leq C e^{-t}.
\end{equation}
Since $\int_E \Delta_E \psi \omega_E^n = 0$, for each time $t$ there exists a point $y(t)$ in $E(z)$ so that $\psi(y(t),t) = 0$. Applying the Mean Value Theorem with \eqref{psiDecay} shows that
\begin{equation}\label{finishMVT}
\vert \Delta_E \psi (x,t) \vert^2_{g_E} = \vert \Delta_E \psi (x,t) - \Delta_E \psi(y(t) ,t) \vert^2_{g_E} \to 0
\end{equation}
as $t \to \infty$. \eqref{convBoundk}, \eqref{laplaceFiberPsi} and  \eqref{finishMVT} show that $e^t g|_{E(z)} \to g_{flat}|_{E(z)}$ in $C^\infty$ on $E(z)$, completing the proof of the main theorem. 
\end{proof}

\section{Acknowledgments}

The author would like to thank his thesis advisor Ben Weinkove for numerous helpful discussions and also Valentino Tosatti for useful comments and suggestions. The author is also grateful for the support and funding of the the San Diego ARCS foundation. 

The contents of this paper will appear in the author's forthcoming PhD thesis.

\bigskip
\noindent
Mathematics Department, University of California, San Diego, 9500 Gilman Drive \#0112, La Jolla CA 92093

\end{document}